\documentclass[11pt]{article}
\usepackage[a4paper, total={6in, 8in}]{geometry}
\usepackage{enumerate}
\usepackage{amsmath}
\usepackage{amssymb,latexsym}
\usepackage{amsthm}
\usepackage{color}
\usepackage{cancel}
\usepackage{graphicx}
\usepackage{cite}
\usepackage{changes}

\makeatletter
\usepackage{comment}
\let\wfs@comment@comment\comment
\let\comment\@undefined

\usepackage{changes}
\let\wfs@changes@comment\comment
\let\comment\@undefined

\newcommand\comment{%
	\ifthenelse{\equal{\@currenvir}{comment}}
	{\wfs@comment@comment}
	{\wfs@changes@comment}%
}

\definechangesauthor[name=Daniele, color=red]{DAN}
\definechangesauthor[name=Marco, color=blue]{MARCO}
\definechangesauthor[name=Matteo, color=violet]{MATTEO}

\usepackage{hyperref}
\usepackage{comment}
\usepackage{mathrsfs}
\usepackage{amscd}

\newtheorem{theorem}{Theorem}[section]

\newtheorem{corollary}[theorem]{Corollary}

\newtheorem{proposition}[theorem]{Proposition}

\newcommand{\cC}{{\mathcal C}}

\newcommand{\cH}{{\mathcal H}}

\newcommand{\F}{{\mathbb F}}

\newcommand{\fq}{{\mathbb F}_{q}}

\title{Minimal codewords in Norm-Trace codes}
\author{Daniele Bartoli\thanks{Dipartimento di Matematica e Informatica, Universit\`a degli Studi di Perugia,  Perugia, Italy. daniele.bartoli@unipg.it}, Matteo Bonini \thanks{Aalborg University, Department of Mathematical Sciences, Aalborg, Denmark. mabo@math.aau.dk} and 
	Marco Timpanella\thanks{Dipartimento di Matematica e Informatica, Universit\`a degli Studi di Perugia,  Perugia, Italy.
		marco.timpanella@unipg.it}
}
\date{ }

\begin{document}
	
	\maketitle
	\begin{abstract}
		In this paper, we consider the affine variety codes obtained evaluating the polynomials  $by=a_kx^k+\dots+a_1x+a_0$, $b,a_i\in\mathbb{F}_{q^r}$, at the affine $\F_{q^r}$-rational points of the Norm-Trace curve. In particular, we investigate the weight distribution and the set of minimal codewords. Our approach, which uses tools of algebraic geometry, is based on the study of the absolutely irreducibility of certain algebraic varieties.
	\end{abstract}
	
	\vspace{0.5cm}\noindent {\bf Keywords}:
	Norm-trace curve; Minimal codewords; Affine variety codes; Weight spectrum
	
	\vspace{0.2cm}\noindent{\bf MSC codes}:
	14G50 - 11T71 - 94B27
	\vspace{0.2cm}\noindent

	\section{Introduction}\label{Sec:Intro}
	
	Affine variety codes \cite{MR1600184} are linear codes obtained evaluating multivariate polynomials at the $\F_q$-rational
	points of a certain affine variety. Since any linear code can be described as an affine variety code (see \cite[Prop 1.4]{MR1600184}), such codes constitute the entire class of linear codes. Even though it is easy to determine the length  and the dimension of an affine variety code, a more difficult task is to provide estimates on the minimum distance, or, more in general, on the weight distribution of the code. Still, computing the planar intersections of the chosen variety with some low-degrees ones is often useful in obtaining information on the weight spectrum and the weight distribution of affine-variety codes, see for example \cite{bartoli2019minimum, couvreur2012dual, marcolla2016small,MR4115116, geil2003codes}.
	
	Given any linear code $C$, another challenging task is the determination of the set of its minimal codewords. For a codeword $c\in C$, the \textit{support} of $c$, denoted by $\mathrm{Supp(c)}$, is the set of its nonzero coordinate positions, and the \textit{weight} of $c$ is ${\rm wt}(c)=\#\mathrm{Supp(c)}$.
	If the support of a codeword $c$  contains the support of another codeword $c^\prime$, then we will say that $c$ \textit{covers} $c^\prime$.
	A codeword $c$ is said to be \textit{minimal} if it covers only the proportional codewords, i.e. if $c^\prime\in C$ is linear independent with $c$, then  $\mathrm{Supp(c^\prime)}\not\subseteq\mathrm{Supp(c)}$. Minimal codewords were employed by Massey  \cite{Massey1999MinimalCA} for the construction of a secret sharing scheme.  For this reason, in recent years,  several papers have been dedicated to the determination of the minimal codewords of a linear code \cite{bartoli2021small,bartoli2021inductive, bonini2021minimal, alfarano2022three,heger2021short,bartoli2022cutting, maji2021one, alfarano2022linear, santonastaso2022subspace, bartoli2021weight}.
	
	In this paper, we give information on the weight distribution and on the minimal codewords of affine variety codes arising from the Norm-Trace curve, as already investigated in literature; see \cite{bonini2020intersections,bonini2022rational,ballico2013duals}. More in detail, throughout the paper we consider the affine variety code $C_{q,r,k}$ obtained evaluating the polynomials  \begin{equation}\label{Eq:curveRaz}
		b y=a_kx^k+\dots+a_1x+a_0,
	\end{equation}
	where $b,a_i\in \mathbb{F}_{q^r}$, at the affine $\mathbb{F}_{q^r}$-rational points of the Norm-Trace curve $\mathcal{N}_{q,r}$, that is the plane curve defined by the affine equation
	\begin{equation}\label{Eq:normtrace}
		x^{\frac{q^r-1}{q-1}}=y^{q^{r-1}}+y^{q^{r-2}}+\ldots+y^q+y.
	\end{equation}
	Note that, up to rescaling, we can assume that the polynomials as in \eqref{Eq:curveRaz} are either of type
	\begin{equation}\label{Eq:curveRaz1}
		y=a_kx^k+\dots+a_1x+a_0,
	\end{equation}
	or
	\begin{equation*}
		a_kx^k+\dots+a_1x+a_0=0.
	\end{equation*}
	In order to obtain information on the weight distribution of the code $C_{q,r,k}$, we deal with the possible intersection patterns of the curve $\mathcal{N}_{q,r}$ and the curves with affine equation \eqref{Eq:curveRaz1}. To do this, our approach is based on the investigation of the absolutely irreducibility of a certain algebraic variety, and therefore it relies on tools of algebraic geometry; see Section \ref{Sec:Preliminaries} for the details. 
	In the last decades, such tools have proved successful in the construction and investigation of many classes of linear codes; see for instance \cite{MR4115116,LT2021,MTZ,MR4149378,MR3775426,TZ}.

	The paper is organized as follows. The prerequisites on  Norm-Trace curves, affine variety codes, and the description of our approach, are give in Section \ref{Sec:Preliminaries}. Section \ref{Sec:Irriducibilità} deals with the absolutely irreducibility of an algebraic variety attached to the problem, and these results are then applied in Section \ref{Sec:Weightdistribution}
	to investigate the weight distribution of the code $C_{q,r,k}$. Finally, in Section \ref{Sec:minimali}, we determine the set of minimal codewords of $C_{q,r,k}$.

	\section{Preliminaries}\label{Sec:Preliminaries}
	In this section, we introduce the notation and terminology that we will use throughout the paper.
	Hereafter, $p$ is a prime and $q=p^m$, where $m$ is a positive integer. Also, $\mathbb{F}_q$ denotes the finite field with $q$ elements. With the symbol $\mathbb{A}^r(\mathbb{F}_q)$ (resp. $\mathbb{P}^r(\mathbb{F}_q)$) we denote the affine (resp. projective) $r$-dimensional space over $\mathbb{F}_q$.
	
	The norm function $\text{N}_{\F_q}^{\F_{q^r}}$ and the trace function $\text{T}_{\F_q}^{\F_{q^r}}$ are the functions from $\F_{q^r}$ to $\F_q$ defined by
	\[
	\text{N}_{\F_q}^{\F_{q^r}}(x)=x^{\frac{q^r-1}{q-1}}=x^{q^{r-1}+q^{r-2}+\dots+q+1}
	\]
	and
	\[
	\text{T}_{\F_q}^{\F_{q^r}}(x)=x^{q^{r-1}}+x^{q^{r-2}}+\dots+x^q+x,
	\]
	respectively.
	When $q$ and $r$ are understood, we will write $\text{N}=\text{N}_{\F_q}^{\F_{q^r}}$ and $\text{T}=\text{T}_{\F_q}^{\F_{q^r}}$.
	
	\subsection{Affine variety codes}
	We introduce now  affine variety codes, see \cite{MR1600184} for further information.
	
	Let $t\ge1$ and consider an ideal $I=\langle g_1,\dots,g_s\rangle$ of $\F_q[x_1,\dots,x_t]$,  $\{x_1^q-x_1,\dots,\, x_t^q-x_t\}\subset I$. 
	The ideal $I$ is zero-dimensional and radical. Let $V(I)=\{P_1,\dots, P_n\}$ be the variety of $I$ and $R=\F_q[x_1,\dots,x_t]/I$.
	
	An affine variety code $C(I,L)$ is the image $\phi(L)$ of $L\subseteq R$, a $\mathbb{F}_q$-vector subspace of $R$ of dimension $r$, given by the isomorphism of $\F_q$-vector spaces $\phi:R\longrightarrow \F_q^n$ that evaluates an element $f\in R$ on $\{P_1,\dots,P_n\}$, i.e. $\phi(f)=(f(P_1),\dots,f(P_n))$.
	
	\subsection{Norm-trace curve}
	
	The Norm-Trace curve $\mathcal{N}_{q,r}$ is the plane curve defined by the affine equation
	\begin{equation*}
		\text{N}_{\F_q}^{\F_{q^r}}(x)=\text{T}_{\F_q}^{\F_{q^r}}(y).
	\end{equation*}
	The equation $	\text{N}_{\F_q}^{\F_{q^r}}(x)=\text{T}_{\F_q}^{\F_{q^r}}(y)$ has precisely $q^{2r-1}$ solutions in $\F_{q^r}^2$, so the curve $\mathcal{N}_{q,r}$ has $q^{2r-1}+1$ rational points: $q^{2r-1}$ of them correspond to affine points, plus a single point at infinity $P_{\infty}$. If $r=2$, $\mathcal{N}_{q,r}$ coincides with  the Hermitian curve, whereas $\mathcal{N}_{q,r}$ is singular in $P_{\infty}$ if $r\ge3$.
	
	Let $C_{q,r,k}$ be the affine variety code obtained evaluating the polynomials  \begin{equation*}
		by=a_kx^k+\dots+a_1x+a_0,
	\end{equation*}
	with $b$ and $a_i$ ranging in $\mathbb{F}_{q^r}$, at the $q^{2r-1}$ affine $\mathbb{F}_{q^r}$-rational points of  $\mathcal{N}_{q,r}$, and $k< q^{r-1}$. Then $C_{q,r,k}$ has length $q^{2r-1}$, dimension $k+1$, and the weight of a codeword  associated to the evaluation of a polynomial $by=f(x)$ as in \eqref{Eq:curveRaz} is given by 
	\[\mathrm{w}(\mathrm{ev}(f))=q^{2r-1}-|\mathcal{N}_{q,r}\cap \mathcal{X} \cap  \mathbb{A}^{2}(\F_{q^{r}})|,\]
	where $\mathcal{X}$ is the curve with affine equation $by-f(x)=0$.
	Therefore, in order to investigate the weight distribution of the code $C_{q,r,k}$, we must study the possible planar intersections in $\mathbb{A}^{2}(\F_{q^{r}})$ between  $\mathcal{N}_{q,r}$ and the (rational) curves whose affine equations are given by \eqref{Eq:curveRaz}.
	Here, by planar intersections (or simply intersections) of two curves lying in the affine space $\mathbb{A}^{2}(\F_{q^{r}})$, we mean the number of points in $\mathbb{A}^{2}(\F_{q^{r}})$ lying on both curves, disregarding multiplicity.

	For the remaining part of this section, we report the  approach used in \cite{bonini2020intersections,bonini2022rational} to deal with this problem.

	In the following we deal with the case $b\neq 0$ in \eqref{Eq:curveRaz}. Substituting $y=f(x)$ as in \eqref{Eq:curveRaz1} in the equation of $\mathcal{N}_{q,r}$, and exploiting the linearity of the trace function, we get
	
	\begin{equation}\label{normtrace:y=A(x)}
		\mathrm{N}(x)=\mathrm{T}(a_{k}x^k)+\dots+\mathrm{T}(a_{2}x^{2})+\mathrm{T}(a_{1}x)+\mathrm{T}(a_{0}).
	\end{equation}
	
	Now, fix a normal basis $\mathcal{B}=\{\alpha,\alpha^q,\dots,\alpha^{q^{r-1}}\}$ of $\F_{q^r}$ over $\F_{q}$ with a suitable $\alpha\in\F_{q^r}$ (see \cite{lidl1997finite} for the details), and let $\Phi_{\mathcal{B}}$ be the canonical vector space isomorphism defined by
	
	$$\Phi_{\mathcal{B}}:(\F_{q})^{r} \longrightarrow \F_{q^{r}}$$
	$$\Phi_{\mathcal{B}}((s_{1},\dots,s_{r}))=s_{1}\alpha+s_{2}\alpha^q+\dots+s_{r}\alpha^{q^{r-1}}.$$
	
	This isomorphism allows us to read the norm $\mathrm{N}$ and the trace $\mathrm{T}$ as maps from $(\F_{q})^{r}$ to $\F_{q}$, by taking $\widetilde{\text{N}}=\text{N}\circ\Phi_{\mathcal{B}}$ and $\widetilde{\text{T}}=\text{T}\circ\Phi_{\mathcal{B}}$.
	Let $\text{T}_i:=\text{\text{T}}(a_ix^i)$  and $\widetilde{\text{T}}_i:=\text{T}_i\circ\Phi_{\mathcal{B}}$, for $1\le i\le k$. Then it is readily seen that $\widetilde{\text{N}}$ and $\widetilde{\text{T}}_i$ are homogeneous polynomials of degree respectively $r$ and $i$ in $\F_q[x_1,\dots,x_{r}]$, $i=0,\dots,k$.
	
	Therefore, we can rewrite \eqref{normtrace:y=A(x)} as
	\begin{equation}
		\label{eq:sup}
		\widetilde{\text{N}}(x_1,\dots,x_{r})=\widetilde{\text{T}}_k(x_1,\dots,x_{r})+\dots+\widetilde{\text{T}}_1(x_1,\dots,x_{r})+\text{T}(a_0).
	\end{equation}
	Equation \eqref{eq:sup} is the equation of a  variety $\mathcal{S}$ defined over $\F_q$. Note that the RHS of \eqref{eq:sup} has degree $r$, and the LHS has degree $k$. By construction, the $\F_q$-rational points of $\mathcal{S}$, correspond to the planar intersections in $\mathbb{A}^2(\F_{q^r})$ between the Norm-Trace curve $\mathcal{N}_{q.r}$ and the rational curve of equation $y=f(x)$, see \cite[Remark 4.1]{bonini2020intersections}.
	
	Let $\mathcal{V}_{k,r}$ be the variety $\psi(\mathcal{S})$, where $\psi$ is the affine change of variables of $\mathbb{A}^r(\overline{\F}_q)$ defined by 
	$$\psi(x_1,\dots,x_{r})=M(x_1,\dots,x_{r})^t=(X_1,\dots,X_r)^t,$$ and $M$ is the non-singular matrix
	\[
	M=\begin{pmatrix}
		\alpha&\alpha^q&\dots&\alpha^{q^{r-1}}\\
		\alpha^q&\dots&\alpha^{q^{r-1}}&\alpha\\
		\vdots &\vdots&\vdots&\vdots \\
		\alpha^{q^{r-1}}&\alpha&\dots&\alpha^{q^{r-2}}
	\end{pmatrix}.
	\]

	Then, the variety $\mathcal{V}_{k,r}$ is defined over $\F_{q^r}$, and it has affine equation $V_{k,r}(X_1,\ldots,X_r)=0$, with
	\begin{equation}\label{Eq:Skr}
		V_{k,r}(X_1,\ldots,X_r)=-\prod_{i=1}^r X_i+\sum_{i=1}^r a_k^{q^{i-1}}X_i^k+\dots+\sum_{i=1}^r a_1^{q^{i-1}}X_i+\mathrm{T}(a_0)
	\end{equation}
	Note that $\psi$ an affine change of variables and thus  preserves the number of absolutely irreducible components of $\mathcal{S}$, and their degrees. This equivalence between $\mathcal{V}_{k,r}$ and $\mathcal{S}$ is crucial in our investigation and in the next sections we will make use a number of times of this link. 
	
	\section{Planar intersections of $\mathcal{N}_{q,r}$ and the curves $y-f(x)=0$}\label{Sec:Irriducibilità}
	
	As it was shown in Section \ref{Sec:Preliminaries}, finding the planar intersections of the norm-trace curve $\mathcal{N}_{q,r}$ and the curves of equation \eqref{Eq:curveRaz1} is equivalent to finding the number of $\fq$-rational points of the $\fq$-rational variety $\mathcal{S}$. Our aim is to prove that $\mathcal{S}$ is absolutely irreducible under certain assumptions on $k$ and $r$, by proving the absolutely irreducibility of $\mathcal{V}_{k,r}$ . Indeed, since  $\psi(x_1,\dots,x_{r})$ preserves the number of absolutely irreducible components of a variety, it follows that if $\mathcal{V}_{k,r}$ is absolutely irreducible the same holds for $\mathcal{S}$. Also, if $\mathcal{S}$ is absolutely irreducible we can apply the Lang-Weil bound to estimate the number of its $\F_q$-rational points.
	
	\begin{theorem}\cite[Lang-Weil bound]{MR65218}\label{Th:LW}
		Let $\mathcal{V}\subset \mathbb{P}^N(\mathbb{F}_q)$ be an absolutely irreducible variety of dimension $n$ and degree $d$. Then there exists a constant $C$ depending only on $N$, $n$, and $d$ such that 
		\begin{equation*}
			\left|\#(\mathcal{V}\cap \mathbb{P}^N(\mathbb{F}_q))-\sum_{i=0}^{n} q^i\right|\leq (d-1)(d-2)q^{n-1/2}+Cq^{n-1}.
		\end{equation*}
	\end{theorem}
	
	Although the constant $C$ was not  computed in \cite{MR65218}, explicit estimates have been provided for instance in  \cite{MR2206396,MR1988974,MR1962145,lidl1997finite,MR2121285,MR0429903} and they have the general shape $C=r(d)$ provided that $q>s(n,d)$, where $r$ and $s$ are polynomials of (usually) small degree. We refer to \cite{MR2206396} for a survey on these bounds. We only include the following result due to Cafure and Matera. 
	\begin{theorem}\cite[Theorem 7.1]{MR2206396}\label{Th:CafureMatera}
		Let $\mathcal{V}\subset\mathbb{A}^N(\mathbb{F}_q)$ be an absolutely irreducible variety defined over $\mathbb{F}_q$ of dimension $n$ and degree $d$. If $q>2(n+1)d^2$, then the following estimate holds:
		$$|\#(\mathcal{V}\cap \mathbb{A}^N(\mathbb{F}_q))-q^n|\leq (d-1)(d-2)q^{n-1/2}+5d^{13/3} q^{n-1}.$$
	\end{theorem}

	We report here some results that we will use to prove the irreducibilty of $\mathcal{V}_{k,r}$, under certain conditions on $k$ and $r$.
	As a corollary of \cite[Lemma 4.15]{Bartoli:2020aa4}, we have the following.
	\begin{proposition}\label{criterio2}
		Let $H$ be an hyperplane of $\mathbb{P}^{r}(\F_{q^r})$ such that $\mathcal{V}_{k,r}\cap H$ is non-repeated and absolutely irreducible. Then $\mathcal{V}_{k,r}$ is absolutely irreducible.
	\end{proposition}
	The following result about the absolutely irreducibility of varieties of Fermat-type is well known and it is a direct consequence of their non-singularity. 
	\begin{proposition}\label{criterio3}
		Let $n,r$ be two positive integers such that $p\nmid n$ and $r\geq 3$. Then, the variety of $\mathbb{P}^{r-1}(\overline{\F}_{q})$ with homogeneous equation
		$$
		a_1X_1^n+a_2X_2^n+\ldots+a_rX_r^n=0,
		$$
		where $a_1,\ldots,a_r\in\overline{\F}_{q}$, 
		is absolutely irreducible.
	\end{proposition}
	
	\begin{proposition}\label{prop1}
		Suppose that $k>r\geq 3$ and $p\nmid k$. Then $\mathcal{V}_{k,r}$ is absolutely irreducible.
	\end{proposition}
	\begin{proof}
		It is readily seen that the homogeneous part in $V_{k,r}$ of the highest degree is 
		$$\sum_{i=1}^r a_k^{q^{i-1}}X_i^k,$$
		which is absolutely irreducible by Proposition \ref{criterio3}. Since $\sum_{i=1}^r a_k^{q^{i-1}}X_i^k=0$ is the intersection between $\mathcal{V}_{k,r}$ and the hyperplane at infinity, it follows that $\mathcal{V}_{k,r}$ is absolutely irreducible by Proposition \ref{criterio2}.
	\end{proof}

	\begin{proposition}
		Suppose that $k=r\geq 4$ and $p\nmid k$. Then $\mathcal{V}_{k,r}$ is absolutely irreducible.
	\end{proposition}
	\begin{proof}
		In this case, the homogeneous part in $V_{k,r}(X_1,X_2,\dots,X_r)$ of the highest degree is 
		$$R(X_1,\ldots,X_r):=-\prod_{i=1}^r X_i+\sum_{i=1}^r a_r^{q^{i-1}}X_i^r.$$
		
		Since $r\geq 4$, the polynomial $$R(0,X_2,\ldots,X_r)=\sum_{i=2}^r a_r^{q^{i-1}}X_i^r $$
		is absolutely irreducible by Proposition \ref{criterio3}, and hence also $R(X_1,\ldots,X_r)$ is absolutely irreducible by Proposition \ref{criterio2}.
		
		Finally, since $R(X_1,\ldots,X_r)=0$ is the intersection between $\mathcal{V}_{k,r}$ and the hyperplane at infinity, by Proposition \ref{criterio2} the claim follows.
	\end{proof}

	\begin{proposition}
		Suppose that $k=r\geq 4$ and $p\mid r$. Then $\mathcal{V}_{k,r}$ is absolutely irreducible.
	\end{proposition}
	\begin{proof}
		Write $r=\bar{r}p^\alpha$, with $p\nmid \bar{r}$. Then $\alpha\geq 1$ and $\bar{r}<r$. The homogeneous part in $V_{k,r}(X_1,X_2,\dots,X_r)$ of the highest degree is 
		$$R(X_1,\ldots,X_r):=-\prod_{i=1}^r X_i+\sum_{i=1}^r a_r^{q^{i-1}}X_i^r=-\prod_{i=1}^r X_i+\left(\sum_{i=1}^r \bar{a}_r^{q^{i-1}}X_i^{\bar{r}}\right)^{p^\alpha},$$
		where $\bar{a}_r^{p^\alpha}=a_r$.
		
		We will prove that $R(X_1,\ldots,X_{r-1},1)=0$ is absolutely irreducible. 
		
		Let $F=\sum_{i=1}^{r-1} \bar{a}_r^{q^{i-1}}X_i^{\bar{r}}$. Observe that $F$ is absolutely irreducible by Proposition \ref{criterio3}. Suppose now that $$R(X_1,\ldots,X_{r-1},1)=G(X_1,X_2,\dots,X_r)H(X_1,X_2,\dots,X_r),$$ where $G(X_1,X_2,\dots,X_r)$ and $H(X_1,X_2,\dots,X_r)$ have the following shape
		
		\[
		G(X_1,X_2,\dots,X_r)=F^{\beta}+G_{\bar{r}\beta-1}+\dots+G_0,
		\]
		\[
		H(X_1,X_2,\dots,X_r)=F^{p^\alpha-\beta}+H_{(p^\alpha-\beta)\bar{r}-1}+\dots+H_0,
		\]
		with $0<\beta<p^{\alpha}$, and $H_i$ and $G_j$ are either homogeneous polynomials of degree $i$ and $j$ respectively, or they are the zero polynomials. Thus
		
		\[
		F^\beta H_{(p^\alpha-\beta)\bar{r}-1}+F^{p^\alpha-\beta}G_{\bar{r}\beta-1}=-\prod_{i=1}^{r-1}X_i.
		\]
		This yields $F\mid \prod_{i=1}^rX_i$, a contradiction. Therefore $R(X_1,\ldots,X_{r-1},1)=0$ is absolutely irreducible and so is $\mathcal{V}_{r,k}$ by Proposition \ref{criterio2}.
	\end{proof}

	\begin{proposition}
		Suppose that $0<k<r$. Then $\mathcal{V}_{k,r}$ is absolutely irreducible.
	\end{proposition}
	\begin{proof}
		If $\mathcal{V}_{k,r}$ is reducible then $V_{k,r}(X_1,X_2,\dots,X_r)$  splits into the product of two polynomials $H$ and $G$ with the following shape,
		\[
		H(X_1,X_2,\dots,X_r)=X_1\dots X_s+H_{s-1}+\dots+H_{0},
		\]
		\[
		G(X_1,X_2,\dots,X_r)=X_{s+1}\dots X_{r}+G_{r-s-1}+\dots+G_{0},
		\]
		where $H_i$ and $G_j$ are either homogeneous polynomials of degree $i$ and $j$ respectively, or they are the zero polynomials, and $1\le s \le r-1$.
		Let $F_u=\sum_{i=1}^ra_u^{q^{i-1}}X_i^u$, then 
		\[
		H(X_1,X_2,\dots,X_r)G(X_1,X_2,\dots,X_r)=X_1\cdot\ldots\cdot X_r+\sum_{u=0}^kF_u.
		\]
		
		Because of the shape of $V_{k,r}(X_1,X_2,\dots,X_r)$, for each $i$ such that $i\geq s+1+k-r$ and $i\leq s-1$, we have that $H_i=0$. For the same reason, for each $j$ such that $j\geq k-s+1$ and $j\leq r-s-1$, $G_{j}=0$.
		
		Now observe that it is not possible that $k-r+s+1<0$ or $k+1-s<0$, otherwise there would exist a variable $X_i$ dividing $H(X_1,X_2,\dots,X_r)$ or $G(X_1,X_2,\dots,X_r)$ (and hence dividing $V_{k,r}$). 
		
		Therefore, the only possibility left is $k-r+s+1\ge 0$ and $k+1-s\ge0$, which gives 
		\[
		F_k(X_1,X_2,\dots,X_r)=X_1\dots X_s\cdot G_{k-s}(X_1,X_2,\dots,X_r)+X_{s+1}\dots X_r\cdot H_{k-r+s}(X_1,X_2,\dots,X_r).
		\]
		Still, this is not possible, since for $X_1=0$ we would have  
		\[
		F_k(0,X_2,\dots,X_r)=\sum_{i=2}^ra_k^{q^{i-1}}X_i^k=H_{k-r+s}(0,X_2,\dots,X_r)\prod_{i=s+1}^rX_i.
		\]
		Clearly, this is impossible by Proposition \ref{criterio3}, as this would imply that $\sum_{i=2}^r a_k^{q^{i-1}}X_i^k$ is divisible by $X_{s+1}\cdot\dots\cdot X_r$.
	\end{proof}

	We recall that by definition of $\mathcal{V}_{k,r}$ and $\mathcal{S}$, these two varieties have the same number of absolutely irreducible components.
	Therefore, as a byproduct of the previous results, together with Theorem \ref{Th:CafureMatera}, we directly obtain the following.
	
	\begin{proposition}\label{Prop:General}
		Let $d=\max{(k,r)}$, and suppose that one of the following cases holds:
		\begin{enumerate}
			\item $k>r$, $p\nmid k$;
			\item $k=r\geq 4$;
			\item $0<k<r$.
		\end{enumerate}
		Then, $\mathcal{S}$ is absolutely irreducible and, if $q>2rd^2$, it contains at least $q^{r-1}-(d-1)(d-2)q^{r-3/2}+5d^{13/3}q^{r-2}$ points in $\mathbb{A}^r(\F_{q})$. 
	\end{proposition}
	
	We finally point out that some results for the case $(k,r)=(3,3)$ and $(k,r)=(3,2)$ can be found in \cite{bonini2022rational} and \cite{bonini2020intersections}), respectively. Unfortunately, it does not seem to be easy to say when $\mathcal{V}_{3,3}$ is irreducible, but when this happens it is possible to give a good estimate on the number of planar intersections between the Norm-Trace curve and rational curves of degree up to three. On the other hand, it is possible to prove (see \cite{bonini2020intersections}) that $\mathcal{V}_{2,3}$ is always absolutely irreducible.
	
	\section{On the weight spectrum of Norm-Trace codes}\label{Sec:Weightdistribution}
	
	Since the codewords of $C_{q,r,k}$ are all given by the evaluations of polynomials of the form $by=f(x)$ as in \eqref{Eq:curveRaz}, their weights are then given by 
	\[\mathrm{w}(\mathrm{ev}(by-f))=q^{2r-1}-|\mathcal{N}_{q,r}\cap \mathcal{X} \cap  \mathbb{A}^{2}(\F_{q^{r}})|,\]
	where $\mathcal{X}$ is the curve with affine equation $by-f(x)=0$. 
	Therefore, an estimate on maximum possible number of $\F_q$-rational planar intersections between $\mathcal{N}_{q,r}$ and the curves $\mathcal{X}$ provides a lower bound on the minimum weight of $C_{q,r,k}$. The case when $b=0$ has already been investigated in \cite{geil2003codes}, while the case $b\ne0$ and $\deg(f)\le3$ can be found in \cite{bonini2020intersections,bonini2022rational}. Therefore, from now on we will focus on the case $b\ne0$ and $\deg(f)>3$.
	
	Classical arguments relying on B\'ezout theorem tell us that the number of planar intersections between the two curves can be bounded by the product of the degrees of $\mathcal{N}_{q,r}$ and $\mathcal{X}$. Then, the maximum number of planar intersection is less than or equal to $k \frac{q^r-1}{q-1}$. Therefore the weight of the codewords of $\mathcal{C}_{q,r,k}$ is at least  $q^{2r-1}-s \frac{q^r-1}{q-1}$, where $s\le k$ is the degree of the polynomial whose evaluation defines the codeword.
	
	Still, this result is not tight and, as a byproduct of the results obtained in the previous section, we can give improvements on $\mathrm{d}(C_{q,r,k})$. 
	
	\begin{corollary}
		Consider the norm-trace curve $\mathcal{N}_{q,r}$ over the field $\F_{q^r}$, with $q$ large enough, and the code $C=C_{q,r,k}$. Suppose also that one of the following conditions holds
		\begin{itemize}
			\item[(a)] $k>r$ and $p\not| k$,
			\item[(b)] $k=r\ge 4$,
			\item[(c)] $0< k< r$.
		\end{itemize}Let $c=\mathrm{ev}(by-f(x))\in C$, then:
		\begin{enumerate}[(i)]
			\item If $b=0$ and $f$ has $s$ distinct roots over $\F_{q^r}$, then $\mathrm{w}(c)=q^{2r-1}-sq^{r-1}$.
			\item If $b\ne 0$ then $\mathrm{w}(c)\ge q^{2r-1}-q^r-5d^{13/3}q^{r-1}-(k-1)(k-2)q^{\frac{r-1}{2}}$
		\end{enumerate}
	\end{corollary}
	
	Notice that the cases $k=r=3$ has been investigated in \cite{bonini2022rational}.

	\section{Minimal codewords in Norm-Trace codes}\label{Sec:minimali}
	
	First, we investigate the case $k=r=2$, in which $\mathcal{N}_{q,r}$ coincides with the Hermitian curve $\cH$ of homogeneous equation
	$$
	x^{q+1}=y^q+y.
	$$
	In this section we provide a complete classification of the minimal codewords of the affine variety code $C$ obtained by evaluating the polynomials of degree $2$ with coefficients in $\mathbb{F}_{q^2}[x,y]$ at the points of $\cH$ in $\mathbb{A}^2(\mathbb{F}_{q^2})$, when $q$ is odd. Observe that such a code $C$ contains in particular each codeword of $C_{q,2,2}$.
	In order to describe the minimal codewords of $C$, we consider the possible planar intersections in $\mathbb{A}^2(\mathbb{F}_{q^2})$ between $\cH$ and the algebraic curves $\cC$ described by a polynomials of degree $2$. 
	
	In the case $\cC$ is irreducible, i.e. it is an irreducible conic, a complete list of the possible planar intersections between $\cH$ and $\cC$, which we report below, has been given for $q$ odd in \cite{donati2009intersection}.
	Here, by subconic of a conic $\cC$ we mean $q+1$ points of $\cC$ lying in a Baer subplane $\mathbb{P}^2(\mathbb{F}_q)$ of $\mathbb{P}^2(\mathbb{F}_{q^2})$.
	
	\begin{proposition}\label{conicheirrid}
		In $\mathbb{P}^2(\mathbb{F}_{q^2})$, $q$ odd, the intersection pattern of $\cH$ and an irreducible conic $\cC$ is one of the following.
		\begin{itemize}
			\item[(i)] $\cH \cap \cC=\emptyset$;
			\item[(ii)] $|\cH \cap \cC|=1$;
			\item[(iii)] $|\cH \cap \cC|=2$;
			\item[(iv)] $|\cH \cap \cC|=q+1$. In particular, $\cH \cap \cC$ is a subconic of $\cC$;
			\item[(v)]  $|\cH \cap \cC|\in \{2q,2q+1,2q+2\}$. In particular, $\cH \cap \cC$ is the union of two subconics of $\cC$ sharing either zero, one, or two points;
			\item[(vi)] $|\cH \cap \cC|\in \{q,q+1,q+2\}$ and meets every subconic of $\cC$ in at most four points;
			\item[(vii)] $q-2\sqrt{q}+2\leq |\cH \cap \cC|\leq q+2\sqrt{q}+2$ and meets every subconic of $\cC$ in at most six points.
		\end{itemize}
	\end{proposition}
	
	If $\cC$ is reducible, the following holds.
	
	\begin{proposition}\label{conicherid}
		In $\mathbb{P}^2(\mathbb{F}_{q^2})$, $q$ odd, the intersection pattern of $\cH$ and a reducible conic $\cC$ is one of the following.
		\begin{itemize}
			\item[(viii)] if $\cC$ is a repeated line $\ell$, then either $|\cH \cap \cC|=1$ and $\ell$ is a tangent to $\cH$, or $|\cH \cap \cC|=q+1$ and $\ell$ is a secant to $\cH$;
			\item[(ix)] $|\cH \cap \cC|=2$ and $\cC$ is the product of two distinct tangents to $\cH$;
			\item[(x)]  $|\cH \cap \cC|\in \{q+1,q+2\}$ and $\cC$ is the product of a tangent to $\cH$ and a secant to $\cH$;
			
			\item[(xi)]  $|\cH \cap \cC|\in \{2q+1,2q+2\}$ and $\cC$ is the product of two distinct secants to $\cH$.
		\end{itemize}
	\end{proposition}
	
	We are now in position to prove the main result of this section. Its proof is based on the  observation that a codeword $c\in C$ associated to the evaluation of a degree $2$ polynomial defining a conic $\cC$ is minimal if and only if it does not exist another conic $\cC^\prime\neq \cC$ such that $\cH\cap \cC\subseteq \cH\cap \cC^\prime$.
	
	\begin{proposition}
		Let $q>7$ be odd. With the notations of Propositions \ref{conicheirrid} and \ref{conicherid}, the minimal codewords of the code $C$ arise from conics whose intersection pattern with $\cH$ is as in $(iv), (v), (vi), (vii), (xi)$.
	\end{proposition}
	\proof
	Among the reducible cases, minimal codewords can only arise from conics that are the product of two distinct secants to $\cH$ (case $(xi)$). Indeed, the intersection patterns of cases $(viii),(ix),(x)$ are strictly contained in the intersection of $\cH$ with two distinct (properly chosen) secant lines, and hence, by the above mentioned observation, the corresponding codewords are not minimal. To prove that a conic $\cC$ as in $(iv), (v), (vi), (vii)$ corresponds to a minimal codeword, assume by way of contradiction that there exists a conic $\cC^\prime$ such that $\cH\cap \cC\subseteq \cH\cap \cC^\prime$. Then $\cC\cap \cC^\prime$ contains $\cH\cap \cC$. However, as $|\cH\cap \cC|>4$ holds in each of the cases $(iv), (v), (vi), (vii)$ for $q>7$, this is a contradiction with the B\'ezout's Theorem stating that $|\cC\cap\cC^\prime|\leq 4$. Finally, it is readily seen that cases $(i),(ii),(iii)$ don't correspond to minimal codewords.
	\endproof

	From now on in this section we assume $k>3$, and we give a description of the minimal codewords of the code $C_{q,r,k}$. 
	
	\begin{proposition}
		Let $k< \#\mathcal{N}_{q,r}(\F_{q^r})$.
		
		The minimal codewords of $C_{q,r,k}$ are the ones generated by the evaluations of polynomials of the shape
		\begin{itemize}
			\item[(i)] $y-f(x)$, $\deg(f)= \overline{k}$, with
			$$q^{r-1}-(\max\{\overline{k},r\}-1)(\max\{\overline{k},r\}-2)q^{r-3/2}+5\max\{\overline{k},r\}^{13/3}q^{r-2}>k,$$ and 
			
			$$\overline{k}>r\textrm{ and }p\nmid \overline{k}, \quad \textrm{ or } \overline{k}=r\geq 4, \quad \textrm{ or } 0<\overline{k}<r;$$
			
			\item[(ii)] $g(x)$, where $g(x)$ is a polynomial of degree $k$ having all distinct roots in $\mathbb{F}_{q^r}$;
			\item[(iii)] $y-\alpha$, with $\alpha\in \mathbb{F}_{q^r}$.
		\end{itemize}

	\end{proposition}

	\begin{proof}
		Consider two codewords $c,c^\prime\in C_{q,r,k}$.
		Recall that the codewords of $ C_{q,r,k}$ are the evaluation of polynomials in the span of the set $\{y,x^{i}\}_{i=0,\dots,k}$ at the  $\F_{q^r}$-rational points of $\mathcal{N}_{q,r}$. Let $F(x,y)$ and $F^\prime(x,y)$ be the polynomials that correspond to $c$ and $c^\prime$, respectively.
		First, we assume $F(x,y)=f(x)$ and $F^\prime(x,y)=y-g(x)$. Then, we claim that the support of $c=\mathrm{ev}(f(x))$ doesn't contain the support of $c^\prime=\mathrm{ev}(y-g(x))$. Indeed, write $f(x)=\prod_{i=1}^{\deg(f)} (x-t_i)$, with $t_i\in \overline{\mathbb{F}}_q$. Then, the zeros of $c$ correspond to all the affine points of $\mathcal{N}_{q,r}$ with coordinates $(t_i,y_i^{(j)})$ such that $t_i\in \mathbb{F}_{q^r}$ and $$\text{N}_{\F_q}^{\F_{q^r}}(t_i)=\text{T}_{\F_q}^{\F_{q^r}}(y_i^{(j)}),$$ for $j=1,\ldots,q^{r-1}$.
		Observe that if $t_i\not\in\mathbb{F}_{q^r}$ for every $i\in\{1,\ldots,\deg(f)\}$, then $c$ is a full-weight codeword and hence it is not minimal. 
		Also, for each $t_i\in \mathbb{F}_{q^r}$, there exists at most a unique $\bar{y}_i$ such that $\bar{y}_i=g(t_i)$ and $(t_i,\bar{y}_i)$ belongs to $\mathcal{N}_{q,r}$. Therefore, the support of $c$ cannot contain the support of $c^{\prime}$.
		
		On the other hand, it is readily seen that if $\deg(g)=\bar{k}>0$, and $\bar{k}$ is as in (i), then the support of $c$ cannot be contained in the support of $c^{\prime}$. Indeed, Proposition \ref{Prop:General} together with the assumption $q^{r-1}-(\max\{\overline{k},r\}-1)(\max\{\overline{k},r\}-2)q^{r-3/2}+5\max\{\overline{k},r\}^{13/3}q^{r-2}>k$, show that the zeros of $c$ cannot contain the zeros of $c^{\prime}$. 
		
		Now, we deal with the case $F(x,y)=y-f(x)$ and $F^\prime(x,y)=y-g(x)$, with $f(x)\neq g(x)$.
		Suppose that the zeros of  $c=\mathrm{ev}(y-f(x))$ are also zeros of $c^\prime=\mathrm{ev}(y-g(x))$, $f(x)\neq g(x)$. Thus, they are also zeros of $\hat{c}=c-c^{\prime}=\mathrm{ev}(f(x)-g(x))$. Then, the argument above applied to $c$ and $\hat{c}$ shows that this case is not possible. 
		
		Assume now that $F(x,y)=f(x)$ and $F^\prime(x,y)=g(x)$, and denote by $\{t_1,\ldots,t_h\}$ and  $\{u_1,\ldots,u_l\} $ the zeros of $f$ and $g$ in $\mathbb{F}_{q^r}$, respectively. Then, it is readily seen that the support of $c=\mathrm{ev}(f(x))$ contains the support of  $c^\prime=\mathrm{ev}(g(x))$ if and only if $\{t_1,\ldots,t_h\}\subset \{u_1,\ldots,u_l\} $. As a direct consequence, the minimal codewords arising from a polynomial of type $f(x)$ must be as in Case (ii).
		
		Finally, let $F(x,y)=y-\alpha$ for a certain $\alpha\in\mathbb{F}_{q^r}$. Then, by the above mentioned arguments, together with the fact that the support of $c$ cannot contain the support of a codeword arising from a polynomial of type $y-\beta$, with $\beta\neq\alpha$, we have that in this case $c$ is minimal.
		
	\end{proof}

	\section*{Acknowledgments}
	This research was supported by the Italian National Group for Algebraic and Geometric Structures and their Applications (GNSAGA - INdAM). The third author is funded by the project ``Metodi matematici per la firma digitale ed il cloud computing" (Programma Operativo Nazionale (PON) “Ricerca e Innovazione” 2014-2020, University of Perugia).
	
	\bibliographystyle{plain}
	\bibliography{biblio.bib}

\end{document}